\documentclass[12pt]{amsart}

\advance\hoffset by -0.5 truecm

\usepackage{mathrsfs}

\usepackage{amssymb}
\usepackage{amscd}
\usepackage{verbatim}
\usepackage{epsfig}
\usepackage{euscript}
\usepackage[colorlinks=true]{hyperref}
\hypersetup{urlcolor=blue, citecolor=red}

\newtheorem{Theorem}{Theorem}[section]

\newtheorem{Lemma}[Theorem]{Lemma}
\newtheorem{Proposition}[Theorem]{Proposition}
\theoremstyle{definition}

\def \R{\mathbb R}
\def \p{\partial}

\def\g{\gamma}
\def \g{\gamma}

\def \<{\langle}
\def \>{\rangle}

\def \tM{\widetilde M}

\def \e{\epsilon}
\def \o{\omega}
\def \l{\lambda}
\def \S{\mathbb S}
\def \gxy{\gamma_{x,y,\lambda,\omega}}
\def \xy{x,y,\lambda,\omega}
\def \bdelta{\boldsymbol{\delta}}

\newcommand{\supp}{\operatorname{supp}}

\newcommand{\Diff}{\operatorname{Diff}}

\begin{document}

\title[The local magnetic ray transform of tensor fields]{The local magnetic ray transform of tensor fields}

\author[Hanming Zhou]{Hanming Zhou}
\address{Department of Pure Mathematics and Mathematical Statistics, University of Cambridge, Cambridge CB3 0WB, UK}
\email{hz318@dpmms.cam.ac.uk}


\begin{abstract}
In this paper we study the local magnetic ray transform of symmetric tensor fields up to rank two on a Riemannian manifold of dimension $\geq 3$ with boundary. In particular, we consider the magnetic ray transform of the combinations of tensors of different orders due to the nature of magnetic flows. We show that such magnetic ray transforms can be stably inverted, up to natural obstructions, near a strictly convex (with respect to magnetic geodesics) boundary point. Moreover, a global invertibility result follows on a compact Riemannian manifold with strictly convex boundary assuming that some global foliation condition is satisfied.
\end{abstract}
\maketitle


\section{Introduction}

Given a Riemannian manifold $(M, g)$ and a magnetic field $\Omega$, which is a closed 2-form, we consider the law of motion described by 
\begin{equation}\label{Mag}
\nabla_{\dot\gamma}\dot\gamma=E(\dot\gamma),
\end{equation}
where $\nabla$ is the Levi-Civita connection of $g$ with the Christoffel symbols $\{\Gamma^i_{jk}\}$
and $E:TM\to TM$ is the \emph{Lorentz force}, which is the bundle map uniquely determined by
\begin{equation*}
\Omega_z(v, w)=\langle E_z(v), w\rangle
\end{equation*}
for all $z\in M$ and $v, w\in T_zM$. A curve $\gamma:\R\to M$, satisfying \eqref{Mag} is called a \emph{magnetic geodesic}. The flow on $TM$ defined by $\phi_t: t\rightarrow (\g(t), \dot\g(t))$ is called a \emph{magnetic flow}. One can check that the generator ${\bf G}_{\mu}$ of the magnetic flow is 
$${\bf G}_{\mu}(z, v)={\bf G}(z,v)+E^j_i(z) v^i\frac{\p}{\p v^j},$$
where ${\bf G}(z, v)=v^i\frac{\p}{\p x^i}-\Gamma^i_{jk} v^j v^k\frac{\p}{\p v^i}$ is the generator of the geodesic flow. 
Note that time is not reversible on the magnetic geodesics, unless $\Omega=0$. When $\Omega=0$ we obtain the ordinary geodesic flow. We call the triple $(M, g, \Omega)$ a \emph{magnetic system}. 

From a dynamical system point of view, the magnetic flow is the Hamiltonian flow of $H(v)=\frac{1}{2}|v|_g^2, \, v\in TM$ w.r.t. the symplectic form $\beta=\beta_0+\pi^*\Omega$, where $\beta_0$ is the canonical symplectic form on $TM$ and $\pi: TM\rightarrow M$ is the canonical projection. 
Thus the magnetic flow preserves the level sets of the Hamiltonian function $H$, i.e. every magnetic geodesic has constant speed. Throughout the paper we fix the energy level $H^{-1}(\frac{1}{2})$, so we only consider the unit speed magnetic geodesics.

Given a magnetic geodesic $\g$ and a smooth function $f$ on $SM$, the unit sphere bundle of $M$, the {\it magnetic ray transform} of $f$ along $\g$ is defined by
$$If(\g)=\int f(\g(t),\dot{\g}(t))\,dt.$$
It is easy to check that the kernel of $I$ contains elements of the following form
$$\{f(z,v)={\bf G}_{\mu}\psi(z,v): \psi\in C^{\infty}(SM), \psi|_{\p SM}=0\}.$$
In applications one often considers ray transforms of $f$ which corresponds to symmetric tensor fields, i.e. $f(z,v)=f_{i_1 \cdots i_m}(z)v^{i_1}\cdots v^{i_m}$, denoted by $I_mf$ for nonnegative integers $m$. 
A basic inverse problem regarding the magnetic ray transform on a compact Riemannian manifold with boundary is to recover the tensor field $f$, up to natural obstructions, from $If(\gamma)$ along all magnetic geodesics $\gamma$ joining boundary points. 

If $\Omega=0$, this reduces to the usual geodesic ray transform of tensor fields, known as the {\it tensor tomography} problem. In this case, the natural elements in the kernel of $I_m$ are of the form $d^s\psi$ where $d^s$ is the symmetric differentiation and $\psi$ is a symmetric $(m-1)$-tensor vanishing on the boundary. These natural elements of the kernel are called {\it potential} tensor fields. So the question is whether the whole kernel consists of purely potential tensors, and when this is the case we say that $I_m$ is {\it s-injective} (when $m=0$, this just means injective). The problem is wide open on compact {\it non-trapping} manifolds with strictly convex boundary. Note that a compact manifold with boundary is non-trapping if every geodesic exits the manifold within a finite time. 

More progresses are made on manifolds under the stronger assumption of being {\it simple}. A compact manifold with boundary is simple if it is simply connected, free of conjugate points and $\p M$ is strictly convex. It is known that $I_0$ \cite{Mu77, MR78} and $I_1$ \cite{AR97} are s-injective on simple manifolds with sharp stability estimates \cite{SU04}. For $I_m$, $m\geq 2$ the tensor tomography problem on simple manifolds is still open in general, except the 2D case. $I_m$ is s-injective on simple surfaces for any $m\geq 2$ \cite{Sh07, PSU13}. In higher dimensions, $I_m$ is s-injective for generic simple metrics including the analytic ones \cite{SU05} and a sharp stability estimate holds \cite{St08}. The equivalence between the s-injectivity of $I_m$ and the surjectivity of its adjoint is known on simple manifolds \cite{PZ15}. See also the recent survey \cite{PSU14} and the references therein. For non-simple manifolds, there are studies under various assumptions \cite{PS88, Sh94, Sh97, Sh02, Da06} and possibly with conjugate points \cite{SU08, SU12, MSU15, HU15} or trapped geodesics \cite{Gu14}. Reconstruction formulas and numerical implementations of the geodesic ray transform on surfaces can be found, e.g., in \cite{PU04, Mo14, GM15}.

For the magnetic ray transform, potential tensors might not stay in the kernel of $I$ (except $I_0$ and $I_1$). Generally the natural elements in the kernel of the magnetic ray transform are combinations of tensors of different orders. For example, the magnetic ray transform of $d^s\beta-E(\beta)+d\varphi=\bold G_\mu (\beta+\varphi)$ always vanishes, where $\beta$ is a 1-form and $\varphi$ is a function on $M$, both vanishing on the boundary. In the current paper, we focus on the magnetic ray transform of tensor fields of orders up to 2. In particular, we are interested in the magnetic ray transform of tensor fields which are sums of 1-forms and symmetric 2-tensors. Note that for the geodesic ray transform, it is unnecessary to consider the combination of 1-forms and 2-tensors, since one can decouple the integral by the fact that geodesic flows are symmetric (or time reversible). 

The tensor tomography problem is closely related to another well-known geometric inverse problem, namely the boundary rigidity problem, which is concerned with the recovery of a Riemannian metric on a compact smooth manifold with boundary from the length data of distance minimizing geodesics connecting boundary points. In particular, the linearization of the boundary rigidity problem is the geodesic ray transform of symmetric 2-tensors. It has been proved that simple surfaces are boundary rigid \cite{PU05}, in higher dimensions generic simple manifolds are boundary rigid \cite{SU05} including the analytic ones. See also recent surveys \cite{Cr04, SU08a, UZ16} and the references therein. There is also a boundary rigidity problem on magnetic systems \cite{DPSU07}, whose linearization is exactly the magnetic ray transform of $h+\beta$, where $h$ is a symmetric 2-tensor and $\beta$ is a 1-form. This provides another motivation for considering such magnetic ray transforms.

A new approach to the tensor tomography problem on compact manifolds of dimension $\geq 3$ with strictly convex boundary has been developed recently in \cite{UV15, SUV14} under a global foliation assumption, based on corresponding local invertibility results. It was also applied to the boundary rigidity problem \cite{SUV16} through a pseudo-linearization argument. As a generalization, we study the local invertibility of the magnetic ray transform of tensor fields in the current paper. We ask the following question: can one recover $f$, up to natural obstructions, near a boundary point $p$ from its integrals $If$ along magnetic geodesics near $p$? By saying magnetic geodesics near $p$ we mean that all magnetic geodesic segments that are completely contained in some small neighborhood $O$ of $p$ with end points on $\p M$ close to $p$, which we call {\it $O$-local} magnetic geodesics, denoted by $\mathcal M_O$. Of course, such a set might be empty if there is no additional geometric assumption of the boundary.




In order to state our main theorems in concrete terms, we describe briefly the setting for our problem. Let $M$ be a compact Riemannian manifold with boundary. Let $z\in\p M$, we say $M$ is \emph{strictly magnetic convex (concave)} at $z$ if 
$$\Lambda(z,v)-\<E_z(v), \nu(z)\>_g > 0\, (< 0)$$ 
for all $v\in S_z(\p M)$, where $\Lambda$ is the second fundamental form of $\p M$, $\nu(z)$ is the inward unit vector normal to $\p M$ at $z$. We can extend $M$ to a larger open manifold $\widetilde M$ and denote the extended metric and magnetic field still by $g$ and $\Omega$. Obviously, magnetic geodesics $\g$ on $M$ can be uniquely extended to a magnetic geodesic on $\widetilde M$, and we still denote it by $\g$. Then intuitively the strict magnetic convexity at $z\in \p M$ means that any magnetic geodesic $\g$ which is tangent to the boundary $\p M$ at $z$ will stay away from $M$ except at $z$ locally.

Now let $\rho\in C^{\infty}(\tM)$ be a boundary defining function of $\p M$, so that $\rho\geq 0$ on $M$. Suppose $\p M$ is strictly magnetic convex at $p\in \p M$, then given a magnetic geodesic $\g$ on $\widetilde M$ with $\g(0)=p$, $\dot{\g}(0)\in S_p\p M$, one has 
\begin{equation}\label{-alpha}
\frac{d^2\rho}{dt^2}(\g(t))|_{t=0}=-\Lambda(p,\dot{\g}(0))+\<E_p(\dot{\g}(0)), \nu(p)\>_g<0.
\end{equation}
Similar to \cite{UV15} we consider the function $\tilde{x}(z)=-\rho(z)-\e |z-p|^2$, where $|\cdot|$ can be taken as the Euclidean norm locally, for some small enough $\e>0$, so that $\tilde x$ is strict magnetic concave from $U_c=\{\tilde{x}> -c\}\subset \widetilde M$ for some sufficiently small $c>0$. For the sake of simplicity, we drop the subscript $c$, i.e. $U_c=U$, and $O=U\cap M$ with compact closure. 

From now on, we assume that $M$ is of dimension $\geq 3$. We first consider a simpler case, namely $f=\beta+\varphi$, where $\beta$ is a 1-form and $\varphi$ is a function. In fact such magnetic ray transform is the linearization of the magnetic boundary rigidity problem in a fixed conformal class. Integrals of such tensors also appear in attenuated ray transforms \cite{HS10, SaU11, PSU12, PSUZ16}. For magnetic systems, a similar weighted magnetic ray transform was considered in \cite{Zh16} for studying the magnetic lens rigidity problem in a fixed conformal class. Here we consider purely the integral of $f$, without extra weights. The global case was considered in \cite{DPSU07} for simple systems. When $f$ is just a function on $M$, i.e. $\beta=0$, the local invertibility of $I$ was studied by the author for even a general family of smooth curves, see the Appendix of \cite{UV15}.

\begin{Theorem}\label{local 1-form}
Let $n=dim M\geq 3$. Assume that $\p M$ is strictly magnetic convex at $p\in\p M$. Given $f\in L^2(T^*\overline O)\times L^2(\overline O)$, there is $p\in H^{1}_{loc}(O)$ with $p|_{O\cap \p M}=0$ such that $f-dp\in L^2_{loc}(T^*O)\times L^2_{loc}(O)$ can be determined from $If$ restricted to $O$-local magnetic geodesics. Moreover, the following stability estimate for $s\geq 0$
$$\|f-dp\|_{H^{s-1}(K)}\leq C\|If\|_{H^s(\mathcal M_O)}$$
holds on any compact subset $K$ of $O$, assuming $f$ is in $H^s$ instead of $L^2$.
\end{Theorem}

Next we consider the local magnetic ray transform of $f=h+\beta$ with $h$ a symmetric 2-tensor and $\beta$ a 1-form. As mentioned above, such ray transforms might find its applications in the boundary rigidity problem for magnetic systems. The global case was considered in \cite{DPSU07} for simple magnetic systems satisfying some curvature assumption or real analytic magnetic systems, later on simple 2D magnetic systems \cite{Ai12}.

\begin{Theorem}\label{local 2-tensor}
Let $n=dim M\geq 3$. Assume that $\p M$ is strictly magnetic convex at $p\in\p M$. Given $f\in L^2(Sym^2T^*\overline O)\times L^2(T^*\overline O)$, there exist $u\in H^{1}_{loc}(T^* O)$ and $p\in H^{1}_{loc}(O)$ with $u|_{O\cap \p M}=0$, $p|_{O\cap \p M}=0$ such that $f-(d^su-E(u)+dp)\in L^2_{loc}(Sym^2T^*O)\times L^2_{loc}(T^*O)$ can be determined from $If$ restricted to $O$-local magnetic geodesics. Moreover, the following stability estimate for $s\geq 0$
$$\|f-(d^su-E(u)+dp)\|_{H^{s-1}(K)}\leq C\|If\|_{H^s(\mathcal M_O)}$$
holds on any compact subset $K$ of $O$, assuming $f$ is in $H^s$ instead of $L^2$.
\end{Theorem} 

Theorem \ref{local 1-form} and \ref{local 2-tensor} generalize the Helgason's type of support theorems for the tensor tomography problem of the geodesic flow in the real-analytic category \cite{K09, KS09} and the smooth category \cite{UV15, SUV14} to the magnetic case. Reconstruction formulas can also be derived in the spirit of \cite[Theorem 4.15]{SUV14}.

As an immediate consequence and application of our local invertibility theorems, we consider the global s-injectivity of the magnetic ray transform on tensors. Given a compact Riemannian manifold $(M,g)$ with smooth boundary and a magnetic field $\Omega$, we say that $M$ can be {\em foliated by strictly magnetic convex hypersurfaces} w.r.t. the magnetic system $(M,g,\Omega)$ if there exist a smooth function $\tau:M\to \mathbb R$ and $a<b$, such that $M\subset \{\tau\leq b\}$, the level set $\tau^{-1}(t)$ is strictly magnetic convex from $\{\tau\leq t\}$ for any $t\in (a,b]$, $d\tau$ is non-zero on these level sets and $\{\tau\leq a\}$ has empty interior. Note that $\p M$ is not necessarily a level set of $\tau$.

\begin{Theorem}\label{global result}
Let $M$ be compact with smooth boundary and $dim M\geq 3$, $\p M$ is strictly magnetic convex. Assume that $M$ can be foliated by strictly magnetic convex hypersurfaces and the set $\{\tau\leq a\}$ is non-trapping.\\
a) Given $f\in T^*M\times C^{\infty}(M)$, if $If\equiv 0$, there exists $p\in C^{\infty}(M)$ with $p|_{\p M}=0$ such that $f=dp$.\\
b) Given $f\in Sym^2T^*M\times T^*M$, if $If\equiv 0$, there exist $u\in T^*M$ and $p\in C^{\infty}(M)$ with $u|_{\p M}=0$, $p|_{\p M}=0$, such that $f=d^su-E(u)+dp$.
\end{Theorem}

The proof of the global result is based on a layer stripping argument similar to that in \cite{UV15, SUV14, PSUZ16}, combined with a regularity result of the solutions of some transport equation w.r.t. the magnetic flow on the unit sphere bundle. A global stability estimate can be derived in a similar way.

In the case of absence of magnetic fields, the foliation condition is an analogue of the Herglotz \cite{He1905} and Wieckert \& Zoeppritz \cite{WZ1907} condition $\frac{\p}{\p r}\frac{r}{c(r)}>0$ for radial symmetric metric $c(r)e$ on a disk with $e$ the Euclidean metric, see also \cite[Section 6]{SUV16}. Examples of manifolds satisfying the foliation conditions are compact submanifolds of complete manifolds with positive curvature \cite{GW76}, compact manifolds with non-negative sectional curvature \cite{Es86}, and compact manifolds with no focal points \cite{RS02}. Our foliation condition defined above is the corresponding version for magnetic systems. It implies the absence of trapped magnetic geodesics in $\{\tau>a\}$, but allows the existence of conjugate points (w.r.t. the magnetic geodesics).   

As mentioned above, the main difference of the magnetic tensor tomography problem comparing with the geodesic case is the coupling of tensors of different orders. Similar to \cite{UV15, SUV14}, we introduce some localized version of $I^*I$ near $p\in \p M$ to fit into Melrose's scattering calculus \cite{Mel94}. However, in addition to the exponential conjugacy that appeared in the geodesic papers, we add an extra pair of conjugacy to address the issue from the coupling of tensors of different orders, see Section 2 and 3 for details.  Another technical difficulty comes up during the decoupling of the effects from tensors of different orders when studying the ellipticity of the localized operator near the artificial boundary $\tilde x=-c$. In particular, the nature of the magnetic flow appears in the symmetric 2-tensor case (Section 3.2), our algebraic argument for the ellipticity of the localized operator is different from \cite{SUV14} and it has potential applications to the boundary rigidity problem for magnetic systems and the invertibility of ray transforms along more general curves.

The paper is organized as follows: In Section 2, we define the localized operators and the proper gauges. Section 3 is devoted to the proof of the ellipticity of the localized operator, up to the gauges, which addresses the key technical issue of the paper. The proofs of Theorem \ref{local 1-form} and \ref{local 2-tensor} are given in Section 4. Finally, we give the proof of Theorem \ref{global result} in Section 5.

\medskip

\noindent{\bf Acknowledgements:} The author wants to thank Prof. Gunther Uhlmann for suggesting this problem and reading an earlier version of the paper. Thanks are also due to Prof. Ting Zhou, part of the work was carried out during the author's visit to her at Northeastern University in 2015. This work was supported by EPSRC grant EP/M023842/1.




\section{The localized operators}

For fixed small $c>0$, let $x=\tilde x+c$, thus $U=\{x>0\}$ with the artificial boundary $x=0$. As what has been done in \cite{UV15}, one can complete $x$ to a coordinate system $(x,y)$ on a neighborhood of $p$, such that locally the metric is of the form $g=dx^2+h(x,y)$ where $h(x,y)$ is a metric on the level sets of $x$. For each point $(x,y)$ we can parameterize magnetic geodesics through this point which are `almost tangent' to level sets of $x$ (these are the curves that we are interested in) by $\lambda\p_x+\omega\p_y\in T\widetilde M$, $\omega\in\mathbb S^{n-2}$ and $\lambda$ is relatively small. Given a magnetic geodesic $\gamma_{x,y,\lambda,\omega}(t)=(x'(t),y'(t))$ with $\gamma_{x,y,\lambda,\omega}(0)=(x,y)$, we define $\alpha(x,y,\lambda,\omega)=\frac{d^2}{dt^2}x'(0)$. In particular, $\alpha(x,y,0,\omega)>0$ for $x$ small by the concavity of $x$. Furthermore, it was shown in \cite{UV15} that there exist $\delta_0>0$ small and $C>0$ such that for $|\l|\leq C\sqrt{x}$ (and $|\l|<\delta_0$), $x'(t)\geq 0$ for $t\in (-\delta_0,\delta_0)$, the magnetic geodesics remain in $\{x\geq 0\}$ at least for $|t|<\delta_0$, i.e. they are $O$-local magnetic geodesics for sufficiently small $c$. Note that \cite{UV15} considers ordinary geodesics, but the settings work for general curves, see the Appendix of \cite{UV15}. 

Our inverse problem is now that assuming $$(If)(\xy)=\int_{\R} f(\gxy(t),\dot{\gamma}_{x,y,\lambda,\omega}(t))\, dt=\int_{\mathbb R}f(x'(t),y'(t),\lambda'(t),\omega'(t))\,dt$$ ($f=\beta+\varphi$ or $h+\beta$) is known for all $\gamma_{x,y,\lambda,\omega}\in \mathcal{M}_O$, we would like to recover $f$ from $If$ up to some gauge. Originally $f$ is defined on $M$, we can extend $f$ by zero to $\widetilde M$ so that the integral is actually defined on a finite interval.

Following the approach of \cite{UV15, SUV14}, let $\chi$ be a smooth non-negative even function on $\mathbb R$ with compact support. Given a function $v$ defined on $\mathcal M_{O}$, or more specifically $\{(x,y,\lambda,\omega): \lambda/x\in \mbox{supp}\,\chi\}$, we define
$$J_0v(x,y)=x^{-2}\int v(x,y,\lambda,\omega)\chi(\lambda/x)\, d\lambda d\omega;$$
$$J_1v(x,y)=\int v(x,y,\lambda,\omega)g_{sc}(\lambda\p_x+\omega\p_y)\chi(\lambda/x)\, d\lambda d\omega;$$
$$J_2v(x,y)=x^{2}\int v(x,y,\lambda,\omega)g_{sc}(\lambda\p_x+\omega\p_y)\otimes g_{sc}(\lambda\p_x+\omega\p_y)\chi(\lambda/x)\, d\lambda d\omega$$
where $g_{sc}$ is a {\it scattering metric}, locally it can be written as $g_{sc}=x^{-4}dx^2+x^{-2}h(x,y)$, here $h$ is the metric on the level sets of $x$. Note that the images of $J_i,\, i=0,1,2$ are functions, one-forms and symmetric $2$-tensors on $U=\{x> 0\}$ respectively.

We denote $W:=\begin{pmatrix}1 & 0\\0 & x^{-1}\end{pmatrix}$, define 
\begin{equation}\label{define A_F}
A_F[\beta, \varphi]=W^{-1}e^{-F/x}\begin{pmatrix} J_1 \\ J_0 \end{pmatrix} I e^{F/x}W \begin{pmatrix} \beta \\ \varphi \end{pmatrix};
\end{equation}
\begin{equation}\label{define B_F}
B_F[h, \beta]=W^{-1}e^{-F/x}\begin{pmatrix} J_2 \\ J_1 \end{pmatrix} I e^{F/x}W \begin{pmatrix} h \\ \beta \end{pmatrix}.
\end{equation}
Comparing with the operators in \cite{UV15, SUV14}, we introduce an additional conjugacy $W^{-1}\,\cdot\,W$ in \eqref{define A_F} and \eqref{define B_F}. The extra conjugacy helps to unify the microlocal properties of the components of $A_F$ and $B_F$ (as matrix operators) respectively, see Section 3. Such idea also appeared in \cite{PSUZ16, Zh16} for weighted X-ray transforms. Obviously $A_F\in \mbox{hom}(T^*_{sc}U\times \underline U ,\,T^*_{sc}U\times \underline U)$ and $B_F\in \mbox{hom}(Sym^2T^*_{sc}U\times T^*_{sc}U,\,Sym^2T^*_{sc}U\times T^*_{sc}U)$, where $\underline U$ is the trivial bundle. The local basis for the scattering cotangent bundle $T^*_{sc}U$ is $\frac{dx}{x^2},\,\frac{dy_1}{x}\cdots \frac{dy_{n-1}}{x}$, and its dual $T_{sc}U$, the scattering tangent bundle, has the local basis $x^2\p x,\,x\p y_1\cdots x\p y_{n-1}$. We will show in the next section that $A_F,\, B_F$ are elliptic scattering pseudodifferential opeartors (see \cite{Mel94, UV15, SUV14} for more details) if one enforces some proper gauge conditions.

For the rest of this section, we study the gauge condition that suits the local magnetic ray transform. Let $\delta$ be the divergence on one-forms, which is the adjoint of $d$ relative to the scattering metric $g_{sc}$. Given a function $\varphi$ and a one-form $\beta$, define ${\bf d}\varphi=\begin{pmatrix} d \\ 0\end{pmatrix}\varphi=[d\varphi, 0]$ and $\boldsymbol{\delta}[\beta,\varphi]=\begin{pmatrix}\delta & 0\end{pmatrix}\begin{pmatrix}\beta \\ \varphi\end{pmatrix}=\delta\beta$, we introduce the conjugated operators
$${\bf d}_F=e^{-F/x}W^{-1}{\bf d}e^{F/x},$$
and $\boldsymbol{\delta}_F=e^{F/x}\boldsymbol{\delta}W^{-1}e^{-F/x}$ its adjoint with respect to the scattering metric $g_{sc}$.

\begin{Lemma}\label{symbol 1}
The principal symbol of $\boldsymbol{d}_F\in \mbox{\normalfont{Diff}}^{1,0}_{sc}(\underline U,\,T^*_{sc}U\times \{0\})$ is $\begin{pmatrix} \xi+iF&\eta\otimes&0 \end{pmatrix}^T$;
while the principal symbol of $\boldsymbol{\delta}_F\in \mbox{\normalfont{Diff}}^{1,0}_{sc}(T^*_{sc}U\times \underline U,\, \underline U)$ is $\begin{pmatrix} \xi-iF & \iota_{\eta} & 0 \end{pmatrix}.$
\end{Lemma}

\begin{proof}
Note that by definition ${\bf d}_F=[e^{-F/x}de^{F/x},0]$, and by \cite[Lemma 3.2]{SUV14} the principal symbol of $e^{-F/x}de^{F/x}$ is 
\begin{equation*}
\begin{pmatrix} \xi+iF \\ \eta\otimes\end{pmatrix}.
\end{equation*}
Thus the principal symbol of ${\bf d}_F$ is 
\begin{equation*}
\begin{pmatrix}
\xi+iF \\
\eta\otimes \\
0
\end{pmatrix}.
\end{equation*}
Since $\boldsymbol{\delta}_F$ is the adjoint of ${\bf d}_F$, its symbol is given by the adjoint of that of the latter with respect to $g_{sc}$, i.e.
\begin{equation*}
\begin{pmatrix}
\xi-iF & \iota_{\eta} & 0
\end{pmatrix}.
\end{equation*}
where $\iota_{\eta}$ is the contraction with $\eta$.
\end{proof}


Now let $d^s$ be the symmetric differentiation acting on one-forms, with the adjoint $\delta^s$ acting on symmetric $2$-tensors with respect to $g_{sc}$. Define ${\bf d}^s=\begin{pmatrix}d^s & 0\\-E & d\end{pmatrix}$ and $\boldsymbol{\delta}^s=\begin{pmatrix}\delta^s & -E^*\\ 0 & \delta\end{pmatrix}$, where $E$ is the Lorentz force. We introduce the operators ${\bf d}^s_F=e^{-F/x}W^{-1}{\bf d}^s W e^{F/x}$ and $\boldsymbol{\delta}^s_F=e^{F/x}W\boldsymbol{\delta}^s W^{-1}e^{-F/x}$, then similar to Lemma \ref{symbol 1} we compute their principal symbols:

\begin{Lemma}\label{symbol 2}
The principal symbol of $\boldsymbol{d}^s_F\in \mbox{\normalfont{Diff}}^{1,0}_{sc}(T^*_{sc}U\times \underline U, \, Sym^2T^*_{sc}U\times T^*_{sc}U)$ is $$\begin{pmatrix} \xi+iF & 0 & 0\\ \frac{1}{2}\eta\otimes & \frac{1}{2}(\xi+iF) & 0\\ \frac{1}{2}\eta\otimes & \frac{1}{2}(\xi+iF) & 0\\ a & \eta\otimes_s & 0\\
0 & 0 & \xi+iF\\ b & 0 & \eta\otimes\end{pmatrix};$$
while the principal symbol of $\boldsymbol{\delta}^s_F\in \mbox{\normalfont{Diff}}^{1,0}_{sc}(Sym^2T^*_{sc}U\times T^*_{sc}U,\, T^*_{sc}U\times\underline U)$ is 
$$\begin{pmatrix} \xi-iF & \frac{1}{2}\iota_{\eta} & \frac{1}{2}\iota_{\eta} & \<a,\cdot\> & 0 & \<b,\cdot\>\\ 0 & \frac{1}{2}(\xi-iF) & \frac{1}{2}(\xi-iF) & \iota_{\eta}^s & 0 & 0\\ 0 & 0 & 0 & 0 &\xi-iF & \iota_{\eta}\end{pmatrix}.$$
Here $a$ is a symmetric 2-tensor and $b$ is a 1-form, both independent of $F$.
\end{Lemma}

\begin{proof}
By definition $${\bf d}^s_F=\begin{pmatrix} e^{-F/x}d^s e^{F/x} & 0 \\ -xE & xe^{-F/x}d e^{F/x}x^{-1} \end{pmatrix},$$
and by \cite[Lemma 3.2]{SUV14} the symbol of $e^{-F/x}d^s e^{F/x}$ is 
$$\begin{pmatrix} \xi+iF & 0 \\ \frac{1}{2}\eta\otimes & \frac{1}{2}(\xi+iF) \\ \frac{1}{2}\eta\otimes & \frac{1}{2}(\xi+iF) \\ a & \eta\otimes_s \end{pmatrix},$$
where $a$ is a suitable symmetric $2$-tensor. On the other hand, the principal symbol of $xe^{-F/x}d e^{F/x}x^{-1}$ ($xe^{-F/x}x^2D_x e^{F/x}x^{-1}=x^2D_x+i(F+x)$, however the term $ix$ is of lower order) is
$$\begin{pmatrix}\xi+iF \\ \eta\otimes\end{pmatrix}.$$
Notice that in the scattering setting, the operator $-xE$ has the form
$$\begin{pmatrix}-xE_x^x & -x^2E^y_{x}\\ -E_{y}^x & -xE_{y}^y\end{pmatrix}$$
(note that $E$ is a $(1,1)$-tensor, locally $E=E^x_x dx\otimes \p_x+E^y_x dx\otimes \p_y+E^x_y dy\otimes \p_x+E^y_y dy\otimes \p_y=E^x_x \frac{dx}{x^2}\otimes x^2\p_x+xE^y_x \frac{dx}{x^2}\otimes x\p_y+x^{-1}E^x_y \frac{dy}{x}\otimes x^2\p_x+E^y_y \frac{dy}{x}\otimes x\p_y$), thus there is a non-trivial contribution from the term $-E_{y}^x$ at the boundary $x=0$, denoted by $b$. Then we combine above arguments to give the principal symbol of ${\bf d}^s_F$. Moreover, the principal symbol of $\boldsymbol{\delta}^s_F$ is the adjoint of the principal symbol of ${\bf d}^s_F$ w.r.t. $g_{sc}$.
\end{proof}

Now we introduce the Witten-type solenoidal gauge condition we will use in the paper in the spirit of \cite{SUV14}. The gauge for the operator $A_F$ is 
$$\boldsymbol{\delta}_F e^{-F/x}W^{-1}[\beta,\varphi]=\boldsymbol{\delta}_F[\beta,\varphi]_F=0;$$
while the gauge for the operator $B_F$ is
$$\boldsymbol{\delta}_F^s e^{-F/x}W^{-1}[h,\beta]=\boldsymbol{\delta}_F^s[h,\beta]_{F}=0.$$ 

\section{Ellipticity up to the gauge}

It is well known, see e.g. \cite{DPSU07}, that the maps
$$\Gamma_+: S\widetilde M\times [0,\infty)\rightarrow [\widetilde M\times\widetilde M; diag],\, \Gamma_+(z,v,t)=(z, |z'-z|, \frac{z'-z}{|z'-z|})$$
and
$$\Gamma_-: S\widetilde M\times (-\infty,0]\rightarrow [\widetilde M\times\widetilde M; diag],\, \Gamma_-(z,v,t)=(z, -|z'-z|, -\frac{z'-z}{|z'-z|})$$
are two diffeomorphisms near $S\widetilde M\times \{0\}$. Here $z'=\gamma_{z,v}(t)$, $[\widetilde M\times\widetilde M; diag]$ is the {\it blow-up} of $\widetilde M$ at the diagonal $z=z'$.

Similar to \cite{UV15}, we can also use $(x,y,|y'-y|,\frac{x'-x}{|y'-y|},\frac{y'-y}{|y'-y|})$ as the local coordinates on $\Gamma_+(\supp\tilde\chi\times[0,\delta_0))$; and analogously for $\Gamma_-(\supp\tilde\chi\times (-\delta_0,0])$ the coordinates are $(x,y,-|y'-y|,-\frac{x'-x}{|y'-y|},-\frac{y'-y}{|y'-y|})$. Indeed, this corresponds to the fact that we are using $(x,y,\l,\o)$ with $\o\in \S^{n-2}$, instead of $S\tM$, to parameterize curves, when $|y'-y|$ is large relative to $|x'-x|$, i.e. in our region of interest.

As we want to study the scattering behavior of the operators $A_F$ and $B_F$ up to the scattering front face $x=0$, we instead apply the scattering coordinates $(x,y,X,Y)$, where 
$$X=\frac{x'-x}{x^2}, \, Y=\frac{y'-y}{x}.$$
Under the scattering coordinates
$$dt\, d\lambda\, d\omega=x^2|Y|^{1-n}J(x,y,X,Y)\,dXdY$$
with $J|_{x=0}=1$.
Note that on the blow-up of the scattering diagonal, $\{X=0,Y=0\}$, in the region $|Y|>\epsilon |X|$, thus on the support of $\chi$
\begin{equation}\label{scattering coordinate}
(x,y,|Y|,\frac{X}{|Y|},\hat Y)\quad \mbox{and} \quad (x,y,-|Y|,-\frac{X}{|Y|},-\hat Y)
\end{equation}
are valid coordinates, $\hat Y=\frac{Y}{|Y|}$, with $\pm |Y|$ being the defining functions of the front face of this blow up. 

It was showed in \cite{UV15, SUV14} that under the coordinates \eqref{scattering coordinate} and the scattering tangent and cotangent bases
\begin{equation}\label{lambda omega}
\begin{split}
& g_{sc}\Big((\lambda\circ\Gamma_{\pm}^{-1})\p_x+(\omega\circ\Gamma_{\pm}^{-1})\p_y\Big)\\
= & x^{-1}\bigg(\Big(\pm\frac{X-\alpha(x,y,\pm x|Y|,\pm \frac{xX}{|Y|},\pm \hat Y)|Y|^2}{|Y|}+x\tilde\Lambda_\pm(x,y,x|Y|,\frac{x
  X}{|Y|},\hat Y)\Big)\,\frac{dx}{x^2}\\
&\qquad\qquad+\Big(\pm\hat Y+x|Y|\tilde\Omega_\pm(x,y,x|Y|,\frac{x
  X}{|Y|},\hat Y)\Big)\,\frac{h(\p_y)}{x}\bigg)
\end{split}
\end{equation}
and
\begin{equation}\label{lambda omega prime}
\begin{split}
& (\lambda'\circ\Gamma_{\pm}^{-1})\p_x+(\omega'\circ\Gamma_{\pm}^{-1})\p_y\\
= & x^{-1}\bigg(\Big(\pm \frac{X+\alpha(x,y,\pm x|Y|,\pm \frac{xX}{|Y|},\pm \hat
  Y)|Y|^2}{|Y|}+x|Y|^2\tilde\Lambda'_\pm(x,y,x|Y|,\frac{x
  X}{|Y|},\hat Y)\Big)\,x^2\p_x\\
&\qquad\qquad+\Big(\pm \hat Y+x|Y|\tilde\Omega'_\pm(x,y,x|Y|,\frac{x
  X}{|Y|},\hat Y)\Big)\,x\p_y\bigg).
\end{split}
\end{equation}
Notice that unlike the geodesic case, $\alpha$ is no longer a quadratic form. From now on we denote $\alpha(x,y,\pm x|Y|,\pm \frac{xX}{|Y|},\pm \hat Y)$ by $\alpha_\pm$ and $\frac{X-\alpha_\pm|Y|^2}{|Y|}$ by $S_\pm$, so $\frac{X+\alpha_\pm|Y|^2}{|Y|}=S_\pm+2\alpha_\pm|Y|$.


\subsection{Ellipticity of $A_F$}

According to the definition \eqref{define A_F} and the expressions \eqref{lambda omega}, \eqref{lambda omega prime}, the Schwartz kernel of $A_F$ at the scattering front face $x=0$, so $\lambda/x=S$, is
\begin{equation*}
K_A(0,y,X,Y)=e^{-FX}|Y|^{1-n}\Big\{\chi(S_+)\begin{pmatrix}A^+_{11} & A^+_{10}\\ A^+_{01} & A^+_{00}\end{pmatrix}+\chi(-S_-)\begin{pmatrix}A^-_{11} & -A^-_{10}\\ -A^-_{01} & A^-_{00}\end{pmatrix}\Big\},
\end{equation*}
where
\begin{equation*}
\begin{split}
A^\pm_{11}&=\Big(S_\pm\frac{dx}{x^2}+\hat Y \frac{dy}{x}\Big)\Big((S_\pm+2\alpha_\pm |Y|)(x^2\p_x)+\hat Y (x\p_y)\Big);\\
A^\pm_{10}&=S_\pm\frac{dx}{x^2}+\hat Y \frac{dy}{x};\\
A^\pm_{01}&=(S_\pm+2\alpha_\pm |Y|)(x^2\p_x)+\hat Y (x\p_y);\\
A^\pm_{00}&=1.
\end{split}
\end{equation*}
When $x=0$, $\alpha_\pm$ is simply $\alpha(0,y,0,0,\pm \hat Y)$. Since $\chi$ is an even function, it is easy to see that the term in the bracket of $K_A$ is even in $(X,Y)$. In particular, $A_F$ is a scattering pseudodifferential operator on $U$ of order $(-1,0)$, i.e. $A_F\in \Psi^{-1,0}_{sc}(U)$, see also \cite[Proposition 3.1]{SUV14}. Generally the Schwartz kernel of a scattering pseudodifferential operator has the form $x^\ell K$ with non-zero $K$ smooth in $(x,y)$ down to $x=0$. For our case, the zero in the superscript of $\Psi^{-1,0}_{sc}$ means exactly that $\ell=0$, while the number $-1$, related to $K$, has the meaning similar to the order of standard pseudodifferential operators. To show that $A_F$ is elliptic up to some gauge, we analyze the behavior of its principal symbol taking values at finite points and fiber infinity of $T_{sc}^*U$, in particular at the scattering front face $x=0$.

\begin{Lemma}\label{A_F at infinity}
For any $F>0$, $A_F$ is elliptic at the fiber infinity of $T_{sc}^*U$ when restricted on the kernel of the standard principal symbol of $\boldsymbol{\delta}_F$.
\end{Lemma}

\begin{proof}
The analysis of the principal symbol of $A_F$ at fiber infinity is quite similar to the standard microlocal analysis of a pseudodifferential operator, i.e. the analysis of the conormal singularity of the standard principal symbol of $A_F$ at the diagonal, $X=Y=0$, see e.g. \cite[Lemma 3.4]{SUV14}.

Under our settings, we need to evaluate the integration of the restriction of the Schwartz kernel $K_A$ to the front face along the orthogonal equatorial sphere corresponding to $\zeta=(\xi,\eta)$, i.e. those $(\tilde S,\hat Y)$ with $\xi\tilde S+\eta\cdot\hat Y=0$. Here $\tilde S$ denotes $X/|Y|$. Notice that for this case the extra vanishing factor $|Y|=0$ in $\chi$ and $A_{ij}$, and the exponential conjugacy (as $X=0$) can be dropped. So by the evenness of $K_A$ the standard principal symbol of $A_F$ is essentially of the following form
\begin{equation*}
|\zeta|^{-1}\int_{\zeta^{\perp}\cap(\mathbb R\times \mathbb S^{n-2})}\chi(\tilde S)\begin{pmatrix} \tilde S \\ \hat Y \\ 1\end{pmatrix}\begin{pmatrix}\tilde S & \hat Y & 1 \end{pmatrix}\,d\tilde S d\hat Y.
\end{equation*} 

Given any non-zero pair $[\beta, \varphi]$, $\beta=(\beta^0,\beta')$, in the kernel of the standard principal symbol of $\boldsymbol{\delta}_F$, i.e. $\xi\beta^0+\eta\cdot \beta'=0$,
$$(\sigma_p(A_F)[\beta,\varphi],[\beta,\varphi])=c|\zeta|^{-1}\int_{\zeta^{\perp}\cap(\mathbb R\times \mathbb S^{n-2})}\chi(\tilde S)\,\Big|\tilde S \beta^0+\hat Y\cdot\beta'+\varphi\Big|^2\,d\tilde S d\hat Y.$$
Now to prove the ellipticity of $A_F$, it suffices to show that there is $(\tilde S,\hat Y)\in \zeta^{\perp}\cap(\mathbb R\times \mathbb S^{n-2})$ such that $\chi(\tilde S)>0$ and $\tilde S\beta^0+\hat Y\cdot\beta'+\varphi\neq 0$. We prove by contradiction, assume that for any $(\tilde S,\hat Y)\in \zeta^{\perp}\cap(\mathbb R\times \mathbb S^{n-2})$ with $\chi(\tilde S)>0$, $\tilde S\beta^0+\hat Y\cdot \beta'+\varphi=0$. Notice that if $\chi(\tilde S)>0$, then $\chi(-\tilde S)>0$, thus $-\tilde S\beta^0-\hat Y\cdot \beta'+\varphi=0$, which implies that $\tilde S\beta^0+\hat Y\cdot \beta'=0$ and $\varphi=0$.

On the other hand, we can find generic $n-1$ elements from the set $\{(\tilde S,\hat Y): \xi\tilde S+\eta\cdot \hat Y=0,\, \chi(\tilde S)>0\}$ (notice that here we need the dimension $n$ be at least $3$, if $n=2$ the set might be empty) with $\tilde S \beta^0+\hat Y\cdot \beta'=0$, by linear algebra, this implies that $\beta=0$ (since $\xi\beta^0+\eta\cdot \beta'=0$), which is a contradiction. This completes the proof.
\end{proof} 

\begin{Lemma}\label{A_F at finite}
For any $F>0$, there exists $\chi=\chi_F\in C^{\infty}_c(\mathbb R)$ such that $A_F$ is elliptic at finite points of $T_{sc}^*U$ when restricted on the kernel of the scattering principal symbol of $\boldsymbol{\delta}_F$.
\end{Lemma}

\begin{proof}
In order to find a suitable $\chi$ to make $A_F$ elliptic acting on the kernel of $\sigma_{sc}(\boldsymbol{\delta}_F)$, we follow the strategy of \cite{UV15}, namely we first do the calculation for a Gaussian function $\chi(s)=e^{-s^2/(2F^{-1}\alpha)}$ with $F>0$. Here $\chi$ does not have compact support, thus an approximation argument will be necessary at the end. The calculation of the Fourier transform of $K_A$ is similar to \cite[Lemma 4.1]{UV15} and \cite[Lemma 3.5]{SUV14}, for the sake of completeness, following we give the main steps. 

Denoting $F^{-1}\alpha_\pm$ by $\mu_\pm$, the $X$-Fourier transform of $K_A$, $\mathcal F_{X}K_A(0,y,\xi,Y)$, with $\chi$ chosen as above, is a non-zero multiple of
\begin{align*}
|Y|^{2-n}\bigg\{\sqrt{\mu_+}e^{i\alpha_+(\xi+iF)|Y|^2}\begin{pmatrix}D_{\nu}(D_{\nu}-2\alpha_+|Y|) & -D_{\nu}\hat Y & -D_{\nu}\\ \hat Y(-D_{\nu}+2\alpha_+|Y|) & \hat Y\cdot \hat Y & \hat Y \\ -D_{\nu}+2\alpha_+|Y| & \hat Y & 1\end{pmatrix}e^{-\mu_+(\xi+iF)^2|Y|^2/2}\\
+\sqrt{\mu_-}e^{i\alpha_-(\xi+iF)|Y|^2}\begin{pmatrix}D_{\nu}(D_{\nu}-2\alpha_-|Y|) & -D_{\nu}\hat Y & D_{\nu}\\ \hat Y(-D_{\nu}+2\alpha_-|Y|) & \hat Y\cdot \hat Y & -\hat Y \\ D_{\nu}-2\alpha_-|Y| & -\hat Y & 1\end{pmatrix}e^{-\mu_-(\xi+iF)^2|Y|^2/2}\bigg\},
\end{align*}
where $D_{\nu}$ is the differentiation with respect to the variable of $\hat{\chi}$, i.e. $-(\xi+iF)|Y|$. Taking the derivatives, by polar coordinates the $Y$-Fourier transform takes the form
\begin{equation*}
\begin{split}
& \int_{\mathbb S^{n-2}}\int_0^{\infty}e^{i|Y|\hat Y\cdot\eta}\times\\
& \bigg\{\sqrt{\mu_+}\begin{pmatrix}i\mu_+(\xi+iF)i\mu_+(\xi-iF)|Y|^2+\mu_+ & i\mu_+(\xi+iF)|Y|\hat Y & i\mu_+(\xi+iF)|Y|\\ i\mu_+(\xi-iF)\hat Y |Y| & \hat Y\cdot \hat Y & \hat Y \\ i\mu_+(\xi-iF)|Y| & \hat Y & 1\end{pmatrix}\\
 & \times e^{-\mu_+(\xi^2+F^2)|Y|^2/2}\\
 & +\sqrt{\mu_-}\begin{pmatrix}i\mu_-(\xi+iF)i\mu_-(\xi-iF)|Y|^2+\mu_- & i\mu_-(\xi+iF)|Y|\hat Y & -i\mu_-(\xi+iF)|Y|\\ i\mu_-(\xi-iF)\hat Y |Y| & \hat Y\cdot \hat Y & -\hat Y \\ -i\mu_-(\xi-iF)|Y| & -\hat Y & 1\end{pmatrix}\\
 & \times e^{-\mu_-(\xi^2+F^2)|Y|^2/2}\bigg\}\,d\hat Y d|Y|.
 \end{split}
\end{equation*}
Since the integrand is invariant under the changes from $|Y|$ to $-|Y|$, $\hat Y$ to $-\hat Y$ (thanks to the evenness from $K_A$), we have that the integral above equals
\begin{equation*}
\begin{split}
 \int_{\mathbb S^{n-2}}&\int_{\mathbb R} e^{i(\hat Y\cdot\eta)t}\sqrt{\mu_+}\\
 &\begin{pmatrix}i\mu_+(\xi+iF)i\mu_+(\xi-iF)t^2+\mu_+ & i\mu_+(\xi+iF)t\hat Y & i\mu_+(\xi+iF)t\\ i\mu_+(\xi-iF)\hat Y t & \hat Y\cdot \hat Y & \hat Y \\ i\mu_+(\xi-iF)t & \hat Y & 1\end{pmatrix}\\
 & \times e^{-\mu_+(\xi^2+F^2)t^2/2}\,d\hat Y dt,
\end{split}
\end{equation*}
which gives a constant multiple of
\begin{equation*}
\begin{split}
\int_{\mathbb S^{n-2}} &\frac{1}{\sqrt{\xi^2+F^2}}\\
& \begin{pmatrix}i\mu_+(\xi+iF)i\mu_+(\xi-iF)D_{\hat Y\cdot \eta}^2+\mu_+ & i\mu_+(\xi+iF)\hat Y D_{\hat Y\cdot\eta} & i\mu_+(\xi+iF)D_{\hat Y\cdot\eta}\\ i\mu_+(\xi-iF)\hat Y D_{\hat Y\cdot\eta} & \hat Y\cdot \hat Y & \hat Y \\ i\mu_+(\xi-iF)D_{\hat Y\cdot\eta} & \hat Y & 1\end{pmatrix}\\
& \times e^{-|\hat Y\cdot \eta|^2/2\mu_+(\xi^2+F^2)}\,d\hat Y\\
=& \int_{\mathbb S^{n-2}} \frac{1}{\sqrt{\xi^2+F^2}}\begin{pmatrix}A_{xx} & A_{xy} & A_{x0}\\ A_{yx} & A_{yy} & A_{y0} \\ A_{0x} & A_{0y} & A_{00}\end{pmatrix} e^{-|\hat Y\cdot \eta|^2/2\mu(\xi^2+F^2)}\,d\hat Y,
\end{split}
\end{equation*}
where
\begin{equation*}
\begin{split}
A_{xx} & =(\xi+iF)(\xi-iF)\frac{|\hat Y\cdot\eta|^2}{(\xi^2+F^2)^2},\\
A_{xy} & =-(\xi+iF)\frac{\hat Y\cdot\eta}{(\xi^2+F^2)}\hat Y,\\
A_{x0} & =-(\xi+iF) \frac{\hat Y\cdot\eta}{(\xi^2+F^2)},\\
A_{yx} & =-\hat Y(\xi-iF)\frac{\hat Y\cdot\eta}{(\xi^2+F^2)},\\
A_{yy} & =\hat Y\cdot \hat Y,\\
A_{y0} & =\hat Y,\\
A_{0x} & =-(\xi-iF)\frac{\hat Y\cdot\eta}{(\xi^2+F^2)},\\
A_{0y} & =\hat Y,\\
A_{00} & =1.
\end{split}
\end{equation*}
Therefore, the scattering principal symbol of $A_F$, $\sigma_{sc}(A_F)$, is a non-zero multiple of
\begin{equation*}
\int_{\mathbb S^{n-2}} \frac{1}{\sqrt{\xi^2+F^2}}\begin{pmatrix} -\frac{(\xi+iF)\hat Y\cdot\eta}{\xi^2+F^2} \\ \hat Y \\ 1\end{pmatrix}\begin{pmatrix}-\frac{(\xi-iF)\hat Y\cdot\eta}{\xi^2+F^2} & \hat Y & 1\end{pmatrix} e^{-|\hat Y\cdot \eta|^2/2\mu(\xi^2+F^2)}\,d\hat Y.
\end{equation*}

Given any non-zero pair $[\beta, \varphi]$, $\beta=(\beta^0,\beta')$, in the kernel of the scattering principal symbol of $\boldsymbol{\delta}_F$, i.e. $(\xi-iF)\beta^0+\eta\cdot \beta'=0$ by Lemma \ref{symbol 1},
\begin{equation*}
\begin{split}
(\sigma_{sc}(A_F) & [\beta,\varphi],[\beta,\varphi])\\
= & \frac{c}{\sqrt{\xi^2+F^2}}\int_{\mathbb S^{n-2}}\Big|-\frac{(\xi-iF)\hat Y\cdot\eta}{\xi^2+F^2} \beta^0+\hat Y\cdot\beta'+\varphi\Big|^2e^{-|\hat Y\cdot \eta|^2/2\mu(\xi^2+F^2)}\,d\hat Y.
\end{split}
\end{equation*}
To prove the ellipticity, it suffices to show that there is $\hat Y$ such that $-\frac{(\xi-iF)\hat Y\cdot\eta}{\xi^2+F^2} \beta^0+\hat Y\cdot\beta'+\varphi\neq 0$. Again, we prove by contradiction, assume that for any $\hat Y\in \mathbb S^{n-2}$, $-\frac{(\xi-iF)\hat Y\cdot\eta}{\xi^2+F^2} \beta^0+\hat Y\cdot\beta'+\varphi$ always vanishes. Then $\frac{(\xi-iF)\hat Y\cdot\eta}{\xi^2+F^2} \beta^0-\hat Y\cdot\beta'+\varphi=0$ too, which implies that $\varphi=0$ and $-\frac{(\xi-iF)\hat Y\cdot\eta}{\xi^2+F^2} \beta^0+\hat Y\cdot\beta'=0$ for all $\hat Y$.

On the other hand, since $(\xi-iF)\beta^0+\eta\cdot \beta'=0$,
$$-\frac{(\xi-iF)\hat Y\cdot\eta}{\xi^2+F^2} \beta^0+\hat Y\cdot\beta'=\frac{1}{\xi^2+F^2}(\eta\cdot \beta')(\hat Y\cdot \eta)+\hat Y\cdot \beta'=0$$ 
for all $\hat Y$. It is not difficult to see that this implies that $\beta'=0$, so $\beta^0=-(\xi-iF)^{-1}\eta\cdot\beta'=0$ too. 
Thus we reach a contradiction as $[\beta,\varphi]$ is a non-zero pair, and this establishes the ellipticity of $A_F$ for Gaussian type $\chi$. 

Finally we pick a sequence ${\chi_n}\in C_c^{\infty}(\R)$ which converges to the Gaussian in Schwartz functions, then the Fourier transforms $\hat{\chi_n}$ converge to $\hat{\chi}$. One concludes that for some large enough $n$, if we use $\chi_n$ to define the operator $A_F$, then its principal symbol is still elliptic as desired.
\end{proof}

Combining Lemma \ref{A_F at infinity} and \ref{A_F at finite} we get the following ellipticity result

\begin{Proposition}\label{A_F elliptic}
For any $F>0$, given $\Omega$ a neighborhood of $O$ in $U$, there exist $\chi\in C^{\infty}_{c}(\mathbb R)$ and $N\in \Psi^{-3,0}_{sc}(U;\underline U,\underline U)$ such that 
$A_F+{\bf d}_F N \boldsymbol{\delta}_F\in \Psi_{sc}^{-1,0}(U; T^*_{sc}U\times \underline{U},T^*_{sc}U\times \underline{U})$ is elliptic in $\Omega$.
\end{Proposition} 

This essentially proves the invertibility of the operator $Ie^{F/x}W$ acting on pairs of functions and one-forms under the gauge condition.

\subsection{Ellipticity of $B_F$}

The analysis of $B_F$ is similar to the case of $A_F$ but more complicated. According to the definition \eqref{define B_F} and the expressions \eqref{lambda omega}, \eqref{lambda omega prime}, the Schwartz kernel of $B_F$ at the scattering front face $x=0$ is
\begin{equation*}
K_B(0,y,X,Y)=e^{-FX}|Y|^{1-n}\Big\{\chi(S_+)\begin{pmatrix}B^+_{22} & B^+_{21}\\ B^+_{12} & B^+_{11}\end{pmatrix}+\chi(-S_-)\begin{pmatrix}B^-_{22} & -B^-_{21}\\ -B^-_{12} & B^-_{11}\end{pmatrix}\Big\},
\end{equation*}
where
\begin{equation*}
\begin{split}
B^\pm_{22}&=\Big(\Big(S_\pm\frac{dx}{x^2}+\hat Y \frac{dy}{x}\Big)\otimes \Big(S_\pm\frac{dx}{x^2}+\hat Y \frac{dy}{x}\Big)\Big)\\
& \quad \quad \quad\Big(\Big((S_\pm+2\alpha_\pm |Y|)(x^2\p_x)+\hat Y (x\p_y)\Big)\otimes \Big((S_\pm+2\alpha_\pm |Y|)(x^2\p_x)+\hat Y (x\p_y)\Big)\Big);\\
B^\pm_{21}&=\Big(\Big(S_\pm\frac{dx}{x^2}+\hat Y \frac{dy}{x}\Big)\otimes \Big(S_\pm\frac{dx}{x^2}+\hat Y \frac{dy}{x}\Big)\Big)\Big((S_\pm+2\alpha_\pm |Y|)(x^2\p_x)+\hat Y (x\p_y)\Big);\\
B^\pm_{12}&=\Big(S_\pm\frac{dx}{x^2}+\hat Y \frac{dy}{x}\Big)\\
& \quad \quad \quad \Big(\Big((S_\pm+2\alpha_\pm |Y|)(x^2\p_x)+\hat Y (x\p_y)\Big)\otimes \Big((S_\pm+2\alpha_\pm |Y|)(x^2\p_x)+\hat Y (x\p_y)\Big)\Big);\\
B^\pm_{11}&=\Big(S_\pm\frac{dx}{x^2}+\hat Y \frac{dy}{x}\Big)\Big((S_\pm+2\alpha_\pm |Y|)(x^2\p_x)+\hat Y (x\p_y)\Big)=A^\pm_{11}.
\end{split}
\end{equation*}
Thus $B_F$ is a scattering pseudodifferential operator of order $(-1,0)$ too. Next we show that $B_F$ is elliptic up to the gauge $\bdelta^s_F$.

\begin{Lemma}\label{B_F at infinity}
For any $F>0$, $B_F$ is elliptic at the fiber infinity of $T_{sc}^*U$ when restricted on the kernel of the standard principal symbol of $\boldsymbol{\delta}^s_F$.
\end{Lemma}

\begin{proof}
Similar to the argument in Lemma \ref{A_F at infinity}, the standard principal symbol of $B_F$ at $\zeta=(\xi,\eta)$ is essentially the following
\begin{align*}
|\zeta|^{-1}\int_{\zeta^{\perp}\cap (\mathbb R\times \mathbb S^{n-2})}\chi(\tilde S)\begin{pmatrix} \tilde S^2 \\ \tilde S\hat Y_1 \\ \tilde S \hat Y_2 \\ \hat Y_1\otimes \hat Y_2 \\ \tilde S \\ \hat Y \end{pmatrix}\begin{pmatrix} \tilde S^2 & \tilde S\hat Y_1 & \tilde S \hat Y_2 & \hat Y_1\otimes \hat Y_2 & \tilde S & \hat Y\end{pmatrix}\,d\tilde S d\hat Y.
\end{align*}
Here subscripts 1 and 2 of $\hat Y$ indicate the position of the factors of a 2-tensor it is acting on.

Given a non-zero pair $[h,\beta]$, $h=(h_{xx},h_{xy},h_{yx},h_{yy})$ with $h_{xy}=h_{yx}^T$ and $\beta=(\beta_x,\beta_y)$, assuming $\sigma_p(\bdelta^s_F)[h,\beta]=0$, i.e. 
\begin{equation}\label{kernel of delta^s_F}
\xi h_{xx}+\eta\cdot h_{xy}=0,\quad \xi h_{xy}+\frac{1}{2}(\eta_1+\eta_2)\cdot h_{yy}=0\quad \mbox{and} \quad \xi\beta_x+\eta\cdot \beta_y=0,
\end{equation}
then 
\begin{align*}
&(\sigma_p ( B_F)[h,\beta], [h,\beta])\\
= & c|\zeta|^{-1}\times\\
& \int_{\zeta^{\perp}\cap (\mathbb R\times \mathbb S^{n-2})}\chi(\tilde S)|\tilde S^2h_{xx}+\tilde S(h_{xy}\cdot \hat Y_1+\hat Y_2\cdot h_{xy})+(\hat Y_1\otimes \hat Y_2)\cdot h_{yy}+\tilde S \beta_x+\hat Y\cdot \beta_y|^2\,d\tilde S d\hat Y.
\end{align*}
Now if the integral equals zero, we get that 
$$\tilde S^2h_{xx}+\tilde S(h_{xy}\cdot \hat Y_1+\hat Y_2\cdot h_{xy})+(\hat Y_1\otimes \hat Y_2)\cdot h_{yy}+\tilde S \beta_x+\hat Y\cdot \beta_y=0$$
for all $(\tilde S,\hat Y)$ satisfying $\xi\tilde S+\eta\cdot \hat Y=0$, $\chi(\tilde S)>0$. Notice that $\chi$ is even, this implies that
\begin{equation}\label{decouple}
\tilde S^2h_{xx}+\tilde S(h_{xy}\cdot \hat Y_1+\hat Y_2\cdot h_{xy})+(\hat Y_1\otimes \hat Y_2)\cdot h_{yy}=0\quad \mbox{and} \quad \tilde S \beta_x+\hat Y\cdot \beta_y=0
\end{equation}
for such $(\tilde S,\hat Y)$. Since $\xi\beta_x+\eta\cdot \beta_y=0$, it is shown in the proof of Lemma \ref{A_F at infinity} that $(\beta_x,\beta_y)=0$.

On the other hand, assume $\tilde S=0$, then by the first equality of \eqref{decouple}, $\hat Y\cdot \eta=0$ implies that $(\hat Y_1\otimes \hat Y_2)\cdot h_{yy}=0$ for all $\hat Y\in \eta^{\perp}\cap \mathbb S^{n-2}$ (notice that $h_{yy}$ is a symmetric $(n-1)\times (n-1)$ matrix). Then to show that $h_{yy}=0$, it suffices to verify that $(\eta\otimes \eta)\cdot h_{yy}=0$. If $\eta=0$ then it's done, so we assume that $\eta\neq 0$. Since $|\tilde S|$ needs to be small to guarantee that $\chi(\tilde S)>0$, we denote $\eta\cdot \hat Y=-\tilde S\xi$ by $\varepsilon$ with $|\varepsilon|\ll 1$, then $\hat Y$ can be decomposed as $\hat Y=\frac{\varepsilon}{|\eta|}\frac{\eta}{|\eta|}+\hat Y^{\perp}$, where $\hat Y^{\perp}$ is the projection of $\hat Y$ in $\eta^{\perp}$. If $\xi=0$, by \eqref{kernel of delta^s_F} $(\eta_1+\eta_2)\cdot h_{yy}=0$, so is $(\eta\otimes\eta)\cdot h_{yy}$. If $\xi\neq 0$, by \eqref{kernel of delta^s_F} again, we have 
$$h_{xy}=-\frac{1}{2\xi}(\eta_1+\eta_2)\cdot h_{yy},\quad h_{xx}=-(\eta\cdot h_{xy})/\xi=\frac{1}{\xi^2}(\eta\otimes \eta)\cdot h_{yy}.$$
Plug above equalities into the first part of \eqref{decouple}, then 
$$\Big(\frac{\tilde S^2}{\xi^2}(\eta\otimes \eta)-\frac{\tilde S}{\xi}(\eta\otimes \hat Y+\hat Y\otimes \eta)+(\hat Y\otimes \hat Y)\Big)\cdot h_{yy}=0,$$
equivalently
$$\Big((\frac{\varepsilon}{\xi^2}\eta+\hat Y)\otimes (\frac{\varepsilon}{\xi^2}\eta+\hat Y)\Big)\cdot h_{yy}=\Big((\varepsilon(\frac{1}{\xi^2}+\frac{1}{|\eta|^2})\eta+\hat Y^{\perp})\otimes (\varepsilon(\frac{1}{\xi^2}+\frac{1}{|\eta|^2})\eta+\hat Y^{\perp})\Big)\cdot h_{yy}=0.$$
Since $(\hat Y^{\perp}\otimes\hat Y^{\perp})\cdot h_{yy}=0$, we have
\begin{equation}\label{varepsilon}
\varepsilon^2(\frac{1}{\xi^2}+\frac{1}{|\eta|^2})^2(\eta\otimes\eta)\cdot h_{yy}=-\varepsilon(\frac{1}{\xi^2}+\frac{1}{|\eta|^2})(\eta\otimes \hat Y^{\perp}+\hat Y^{\perp}\otimes \eta)\cdot h_{yy}.
\end{equation}
Notice that for fixed $\hat Y^{\perp}$ and $\varepsilon\neq 0$, $\hat Y=-\frac{\varepsilon\eta}{|\eta|^2}+\hat Y^{\perp}$ will also work for above equation. Thus both sides of \eqref{varepsilon} vanish, in particular $(\eta\otimes \eta)\cdot h_{yy}=0$ and $(\eta\otimes Y+Y\otimes \eta)\cdot h_{yy}$ for any $Y\in \eta^\perp$. Above argument means that $(\hat Y\otimes \hat Y)\cdot h_{yy}=0$ for all $\hat Y\in \mathbb S^{n-2}$. Taking into account the symmetricity of $h_{yy}$, it has to be zero.

Since $h_{yy}=0$, by \eqref{kernel of delta^s_F} if $\xi\neq 0$, we have $h_{xy}=0$ and $h_{xx}=0$. If $\xi=0$, then $\tilde S^2h_{xx}+\tilde S(h_{xy}\cdot \hat Y_1+\hat Y_2\cdot h_{xy})=0$, thus $\tilde S h_{xx}+h_{xy}\cdot \hat Y+\hat Y\cdot h_{xy}=0$ when $\tilde S\neq 0$ small, for any $\hat Y\in \eta^{\perp}\cap \mathbb S^{n-2}$. Take $\tilde S_i\neq 0$ with $\chi(\tilde S_i)>0$, $i=1,2$, then $(\tilde S_1-\tilde S_2)h_{xx}=0$ which implies that $h_{xx}=0$ and $h_{xy}\cdot \hat Y=0$ for all $\hat Y\in \eta^{\perp}\cap \mathbb S^{n-2}$. However, since $\eta\cdot h_{xy}=0$, we get $h_{xy}=0$. Thus $h=(h_{xx},h_{xy},h_{yy})=0$, i.e. $[h,\beta]=0$, which is a contradiction. This proves the lemma.
\end{proof}

\begin{Lemma}\label{B_F at finite}
There exists $F_0>0$, for any $F>F_0$, there is $\chi=\chi_F\in C^{\infty}_c(\mathbb R)$ such that $A_F$ is elliptic at finite points of $T_{sc}^*U$ when restricted on the kernel of the scattering principal symbol of $\boldsymbol{\delta}^s_F$.
\end{Lemma}

\begin{proof}
If $\chi$ is a Gaussian function, i.e. $\chi(s)=e^{-s^2/2F^{-1}\alpha}$, by a computation similar to that of Lemma \ref{A_F at finite} we get that the scattering principal symbol of $B_F$ is a non-zero multiple of
\begin{equation*}
\int_{\mathbb S^{n-2}}\frac{1}{\sqrt{\xi^2+F^2}}\begin{pmatrix} \bar\theta_2\\ \hat Y_1\bar\theta_1 \\ \hat Y_2\bar\theta_1 \\ \hat Y_1\hat Y_2 \\ \bar\theta_1 \\ \hat Y \end{pmatrix}\begin{pmatrix}\theta_2 & \theta_1\hat Y_1 & \theta_1 \hat Y_2 & \hat Y_1\hat Y_2 & \theta_1 & \hat Y\end{pmatrix}e^{-|\hat Y\cdot \eta|^2/2F^{-1}\alpha(\xi^2+F^2)}\,d\hat Y,
\end{equation*}
where $\theta_1=-\frac{\xi-iF}{\xi^2+F^2}(\hat Y\cdot \eta)$ and $\theta_2=\frac{(\xi-iF)^2}{(\xi^2+F^2)^2}(\hat Y\cdot \eta)^2+2i\alpha \frac{\xi-iF}{\xi^2+F^2}=\theta_1^2+2i\alpha  \frac{\xi-iF}{\xi^2+F^2}$.

Given a nonzero pair $[h,\beta]$ in the kernel of the scattering principal symbol of $\bdelta^s_F$, by Lemma \ref{symbol 2}, 
\begin{equation}\label{kernel of delta^s_F 2}
(\xi-iF)h_{xx}+\eta\cdot h_{xy}+a\cdot h_{yy}+b\cdot \beta_y=0,\quad (\xi-iF)h_{xy}+\frac{1}{2}(\eta_1+\eta_2)\cdot h_{yy}=0
\end{equation}
and 
\begin{equation}\label{kernel of delta^s_F 3}
(\xi-iF)\beta_x+\eta\cdot \beta_y=0.
\end{equation}
Then
\begin{align*}
(\sigma_{sc} &(B_F)[h,\beta],[h,\beta])=\frac{c}{\sqrt{\xi^2+F^2}}\times \\
& \int_{\mathbb S^{n-2}}|\theta_2h_{xx}+2\theta_1\hat Y\cdot h_{xy}+(\hat Y\otimes \hat Y)\cdot h_{yy}+\theta_1\beta_x+\hat Y\cdot \beta_y|^2 e^{-|\hat Y\cdot \eta|^2/2F^{-1}\alpha(\xi^2+F^2)}\,d\hat Y.
\end{align*}
If the lemma is not true, then for any $N>0$, there is $F>N$ such that above integral vanishes for some nonzero $[h,\beta]$ in the kernel of $\sigma_{sc}(\bdelta^s_F)$, we get that $\theta_2h_{xx}+2\theta_1\hat Y\cdot h_{xy}+(\hat Y\otimes \hat Y)\cdot h_{yy}+\theta_1\beta_x+\hat Y\cdot \beta_y=0$ for all $\hat Y\in \mathbb S^{n-2}$. Note that $\theta_1(-\hat Y)=-\theta_1(\hat Y)$. On the other hand, by \eqref{-alpha} it is not difficult to see that for magnetic geodesics $\alpha(\hat Y)=d^2x/dt^2|_{t=0}=\alpha^+(\hat Y)+\alpha^-(\hat Y)$ with $\alpha^+$ a positive definite quadratic form (similar to the geodesic case) and $\alpha^-$ a $1$-form (related to $E$). Thus $\theta_2(-\hat Y)=\theta_1^2(\hat Y)+2i(\alpha^+(\hat Y)-\alpha^-(\hat Y))\frac{\xi-iF}{\xi^2+F^2}$, and
\begin{equation}\label{decouple 2}
\begin{split}
& \Big(\theta_1^2(\hat Y)+2i\alpha^+(\hat Y)\frac{\xi-iF}{\xi^2+F^2}\Big)h_{xx}+2\theta_1(\hat Y)\hat Y\cdot h_{xy}+(\hat Y\otimes \hat Y)\cdot h_{yy}=0,\\
& 2i\alpha^-\cdot \hat Y\frac{\xi-iF}{\xi^2+F^2}h_{xx}+\theta_1(\hat Y)\beta_x+\hat Y\cdot \beta_y=0
\end{split}
\end{equation}
for all $\hat Y$. In other words, there exist $\{F_k\}_{k=1}^{\infty}$, $F_k\to +\infty$ as $k\to \infty$, and $\{[h^k,\beta^k]\}_{k=1}^{\infty}$, $[h^k,\beta^k]$ in the kernel of $\sigma_{sc}(\bdelta^s_{F_k})$ and nonzero, such that \eqref{decouple 2} holds for each pair $(F_k,[h^k,\beta^k])$. 

First we claim that for large enough $k$, $h^k_{yy}\neq 0$. If not, then there exists a subsequence $\{F_{n_k}, [h^{n_k},\beta^{n_k}]\}$ such that $h^{n_k}_{yy}=0$ for all $n_k$. Then by \eqref{kernel of delta^s_F 2} $h^{n_k}_{xy}=0$ and $h^{n_k}_{xx}=-b\cdot\beta^{n_k}_y/(\xi-iF_{n_k})$. So by \eqref{kernel of delta^s_F 3} and the second equation of \eqref{decouple 2}, 
$$\Big(-2i\frac{b\cdot\beta_y^{n_k}}{\xi^2+F_{n_k}^2}\alpha^-+\frac{\eta\cdot\beta^{n_k}_y}{\xi^2+F_{n_k}^2}\eta+\beta^{n_k}_y\Big)\cdot \hat Y=0$$
for all $\hat Y\in \mathbb S^{n-2}$, i.e. 
\begin{equation}\label{beta_y}
-2i\frac{b\cdot\beta_y^{n_k}}{\xi^2+F_{n_k}^2}\alpha^-+\frac{\eta\cdot\beta^{n_k}_y}{\xi^2+F_{n_k}^2}\eta+\beta^{n_k}_y=0.
\end{equation}
If $\beta_y^{n_k}=0$, then by \eqref{kernel of delta^s_F 3} $\beta_x^{n_k}=0$ and $h_{xx}^{n_k}=0$, i.e. $[h^{n_k},\beta^{n_k}]=0$, which is a contradiction. Thus we can assume that $\beta_y^{n_k}$ has unit norm for all $n_k$ (notice that at a fixed point the geometry is trivial). Let $F_{n_k}\to +\infty$, then by \eqref{beta_y} $\beta_y^{n_k}\to 0$, which is again a contradiction.


Now we can assume that $h_{yy}^k\neq 0$ for all $k$. 
By \eqref{kernel of delta^s_F 2} and \eqref{kernel of delta^s_F 3}, for any $k$
\begin{equation*}
\begin{split}
& h^k_{xy}=-\frac{\eta_1+\eta_2}{2(\xi-iF_k)}\cdot h^k_{yy}, \\
& h^k_{xx}=-\frac{\eta\cdot h^k_{xy}+a\cdot h^k_{yy}+b\cdot \beta^k_y}{\xi-iF_k}=\frac{\eta\otimes \eta-(\xi-iF_k)a}{(\xi-iF_k)^2}\cdot h^k_{yy}-\frac{b}{\xi-iF_k}\cdot \beta^k_y,\\
& \beta_x^k=-\frac{\eta\cdot \beta_y^k}{\xi-iF_k}.
\end{split}
\end{equation*}
Plugging above equalities into \eqref{decouple 2} to get
\begin{equation}\label{zero for all k 1}
\begin{split}
\Big(\frac{(\hat Y\cdot \eta)^2+2i\alpha^+(\xi+iF_k)}{(\xi^2+F_k^2)^2} & (\eta\otimes \eta-(\xi-iF_k)a)+\frac{\hat Y\cdot\eta}{\xi^2+F_k^2}(\eta\otimes\hat Y+\hat Y\otimes\eta)+\hat Y\otimes\hat Y\Big)\cdot h^k_{yy}\\
& -\frac{\xi-iF_k}{(\xi^2+F_k^2)^2}\Big((\hat Y\cdot \eta)^2+2i\alpha^+(\xi+iF_k)\Big)b\cdot \beta^k_y=0,
\end{split}
\end{equation}
and
\begin{equation}\label{zero for all k 2}
2i(\alpha^-\cdot \hat Y)\frac{\eta\otimes \eta-(\xi-iF_k)a}{(\xi^2+F_k^2)(\xi-iF_k)}\cdot h^k_{yy}+\Big(\hat Y+\frac{\hat Y\cdot \eta}{\xi^2+F_k^2}\eta-\frac{2i(\alpha^-\cdot \hat Y)}{\xi^2+F_k^2}b\Big)\cdot\beta^k_y=0.
\end{equation}

If there is a subsequence of $\{\beta_y^{n_k}\}_{n_k\to \infty}$ such that $\beta_y^{n_k}=0$ for all $n_k$, since $h_{yy}^k\neq 0$, we may assume that $h_{yy}^{n_k}$ has unit norm for all $n_k$. Thus there exists further a subsequence $\{h_{yy}^{n'_k}\}_{n'_k\to \infty}$ of $\{h_{yy}^{n_k}\}$ and $h_{yy}^{\infty}$ satisfying $h_{yy}^{n'_k}\to h_{yy}^{\infty}$, $F_{n'_k}\to +\infty$ as $n'_k\to \infty$. As $(\xi,\eta)$ is a finite point, we take the limit of \eqref{zero for all k 1} as $n'_k\to \infty$ to get that 
$$(\hat Y\otimes \hat Y)\cdot h^{\infty}_{yy}=0, \quad\forall\, \hat Y\in \mathbb S^{n-2}.$$
Since $h^{\infty}_{yy}$ is a symmetric tensor, above equality forces it to be zero. However, since $h^{n_k}_{yy}$ has unit norm, the limit $h^{\infty}_{yy}$ can not be zero, we reach a contradiction. 

So we can assume that $h_{yy}^k\neq 0$ and $\beta_y^k\neq 0$ for any $k$. Let $c_k=\max\{\|h^k_{yy}\|, \|\beta^k_y\|\}>0$, consider the sequence $\{[h^k/c_k,\beta^k/c_k]\}$, we still denote the new sequence by $\{[h^k,\beta^k]\}$, thus $\|h^k_{yy}\|\leq 1$ and $\|\beta^k_y\|\leq 1$. Then there exists a subsequence $\{(h^{n_k},\beta^{n_k})\}_{n_k\to\infty}$ such that $h^{n_k}_{yy}\to h^{\infty}_{yy}$, $\beta^{n_k}_y\to \beta^{\infty}_{y}$, $F_{n_k}\to +\infty$ as $n_k\to \infty$. Now we take the limits of \eqref{zero for all k 1} and \eqref{zero for all k 2} with respect to the subsequence as $n_k\to \infty$ to get that 
$$(\hat Y\otimes\hat Y)\cdot h^{\infty}_{yy}=0,\quad \hat Y\cdot \beta_y^{\infty}=0, \quad \forall \hat Y\in \mathbb S^{n-2}.$$
Again this implies that $h^{\infty}_{yy}=0$ and $\beta^{\infty}_y=0$. However for each $n_k$, either $\|h^{n_k}_{yy}\|=1$ or $\|\beta^{n_k}_y\|=1$, so $h^{\infty}_{yy}$ and $\beta^{\infty}_y$ can not both vanish. This is a contradiction too, thus our assumption for the contradiction argument is not true, i.e. there is some $F_0>0$ such that the lemma holds for Gaussian like $\chi$. Then we apply an approximation argument to complete the proof.
\end{proof}

\noindent {\it Remark:} The algebraic argument of the proof of Lemma \ref{B_F at finite} is different from the one of \cite{SUV14} as the magnetic case is more complicated than the geodesic case due to the coupling of tensors of different orders. In particular, $\alpha$ is no longer an even function of $\hat Y$ as in the geodesic case, which is the reason why we consider $h_{yy}$ and $\beta_y$ together in the main argument. On the other hand, our idea might work for the tensor tomography problem along general smooth curves, since generally one can decompose $\alpha$ into the even and odd parts with $\alpha=\alpha^++\alpha^-$, where $\alpha^+(\hat Y)=(\alpha(\hat Y)+\alpha(-\hat Y))/2$ and $\alpha^-(\hat Y)=(\alpha(\hat Y)-\alpha(-\hat Y))/2$.

\medskip

Similar to Proposition \ref{A_F elliptic}, we have the following result for $B_F$.

\begin{Proposition}\label{B_F elliptic}
There exists $F_0>0$ such that for any $F>F_0$, given $\Omega$ a neighborhood of $O$ in $U$, there exist $\chi\in C^{\infty}_{c}(\mathbb R)$ and $N\in \Psi^{-3,0}_{sc}(U; T^*_{sc}U\times\underline{U},T^*_{sc}U\times\underline{U})$ such that 
$B_F+{\bf d}^s_F N \boldsymbol{\delta}^s_F\in \Psi_{sc}^{-1,0}(U;Sym^2T^*_{sc}U\times T^*_{sc}U,Sym^2T^*_{sc}U\times T^*_{sc}U)$ is elliptic in $\Omega$.
\end{Proposition}


\section{Proofs of the main local results}

Now we rephrase the invertibility results of Section 3 in a gauge free way. This part is similar to \cite[Section 4]{SUV14}, the key ingredient is the local invertibility of some Witten-type Dirichlet Laplacian.

\subsection{Proof of Theorem \ref{local 1-form}}

Note that if the `solenoidal Witten Laplacian' $\Delta_{F}=\bdelta_{F}\bold d_{F}$ is invertible with the Dirichlet boundary condition, we can decompose $f_F:=[\beta,\varphi]_F=e^{-F/x}W^{-1}[\beta,\varphi]$ into
$$f_F=\mathcal S_{F}f_F+\mathcal P_{F}f_F,$$
where $\mathcal P_{F}=\bold d_{F}\Delta_{F}^{-1}\bdelta_{F}$. Thus we denote $\mathcal P_{F}f_F$ by $\bold d_{F}p_F=W^{-1}e^{-F/x}\bold dp$ with $p|_{\p O\cap \p M}=0$, then given $f=[\beta, \varphi]$
$$If=I(f-\bold dp)=I(e^{F/x}W(f_F-\bold d_{F}p_F))=I(e^{F/x}W\mathcal S_{F}f_F).$$
Notice that $\bdelta_{F}(\mathcal S_{F}f_F)=0$, by Proposition \ref{A_F elliptic} in $O$, $\mathcal S_{F}f_F$ or equivalently $e^{F/x}W\mathcal S_{F}f_F=f-\bold dp$ can be stably determined by $If=I(e^{F/x}W\mathcal S_{F}f_F)$, see \cite[Theorem 4.15]{SUV14}, also \cite[Sect. 3.7]{UV15} for the function case. Generally the stability estimate by ellipticity has an error term, however for the local problem the error term is relatively small and can be absorbed to produce the full invertibility, see \cite[Sect. 2]{UV15}. This proves Theorem \ref{local 1-form}. 
So one just needs to show that $\Delta_{F}$ is invertible with the Dirichlet boundary condition, however this is immediate from the argument of \cite[Section 4]{SUV14}. Note that by the definition, $\Delta_F$ is 
the same as the Witten Laplacian of functions in \cite{SUV14}.

\subsection{Proof of Theorem \ref{local 2-tensor}}

Similar to the argument of Section 4.1, if the Witten Laplacian $\Delta_F^s=\bdelta^s_F\bold d^s_F$ is invertible with the Dirichlet boundary condition, let $f=[h,\beta]$, then by Proposition \ref{B_F elliptic} there are some 1-form $u$ and function $p$ with $u|_{\p O\cap \p M}=0$, $p|_{\p O\cap \p M}=0$ such that $f-\bold d^s[u,p]$ can be stably determined by $If$. Notice that by Lemma \ref{symbol 2}, the principal symbol of $\Delta^s_{F}$ is
\begin{equation*}
\begin{pmatrix}\<\xi\>^2+\frac{1}{2}|\eta|^2 & \frac{1}{2}(\xi+iF)\iota_{\eta} & 0 \\ \frac{1}{2}(\xi-iF)\eta\otimes & \frac{1}{2}\<\xi\>^2+|\eta|^2 & 0 \\ 0 & 0 & \<\xi\>^2+|\eta|^2 \end{pmatrix}+\begin{pmatrix} \<a,\cdot\>a+\<b,\cdot\>b & \<a,\cdot\>\eta\otimes_s & \<b,\cdot\>\eta\otimes \\ \iota_{\eta}^s a & 0 & 0 \\ \iota_\eta b & 0 & 0 \end{pmatrix},
\end{equation*}
where $\<\xi\>=\sqrt{\xi^2+F^2}$. It is easy to check that the first part of the symbol has a lower bound $O(\xi^2+F^2+|\eta|^2)$, by taking $F$ large enough, it can absorb the second part of the symbol which is independent of $F$. Thus $\Delta^s_F$ is elliptic for large $F$. Moreover, let $\nabla_F=e^{-F/x}\nabla e^{F/x}$ with $\nabla$ being the gradient with respect to the scattering metric $g_{sc}$, we define $\nabla_F^s[\beta,\varphi]:=[\nabla_F\beta, \sqrt{2}\nabla_F\varphi]$, which has the following principal symbol
\begin{equation*}
\begin{pmatrix} \xi+iF & 0 & 0 \\ \eta\otimes & 0 & 0 \\ 0 & \xi+iF & 0\\ 0 & \eta\otimes & 0 \\ 0 & 0 & \sqrt{2}(\xi+iF) \\ 0 & 0 & \sqrt{2}\eta\otimes \end{pmatrix}. 
\end{equation*}
So the principal symbol of $(\nabla^s_F)^*$, the adjoint under the scattering metric $g_{sc}$, is 
\begin{equation*}
\begin{pmatrix} \xi-iF & \iota_\eta & 0 & 0 & 0 & 0 \\ 0 & 0 & \xi-iF & \iota^s_\eta & 0 & 0 \\ 0 & 0 & 0 & 0 & \sqrt{2}(\xi-iF) & \sqrt{2}\iota_\eta \end{pmatrix}
\end{equation*}
and the principal symbol of $(\nabla^s_F)^*\nabla^s_F$ is 
\begin{equation*}
\begin{pmatrix} \<\xi\>^2+|\eta|^2 & 0 & 0 \\ 0 & \<\xi\>^2+|\eta|^2 & 0 \\ 0 & 0 & 2(\<\xi\>^2+|\eta|^2) \end{pmatrix}.
\end{equation*}
On the other hand, applying Lemma \ref{symbol 1} again, we get the principal symbol of $\bold d_F\bdelta_F$
\begin{equation*}
\begin{pmatrix} \<\xi\>^2 & (\xi+iF)\iota_\eta & 0 \\ (\xi-iF)\eta\otimes & |\eta|^2 & 0 \\ 0 & 0 & 0 \end{pmatrix}.
\end{equation*}
Therefore, $\Delta^s_F=\frac{1}{2}(\nabla_F^s)^*\nabla^s_F+\frac{1}{2}\bold d_F\bdelta_F+R+\tilde R$ with $R\in \Diff^1_{sc}(T^*_{sc}U\times \underline U, T^*_{sc}U\times \underline U)$ given by the second part of the principal symbol of $\Delta^s_F$ and $\tilde R\in x \Diff^1_{sc}(T^*_{sc}U\times \underline U, T^*_{sc}U\times \underline U)$ containing all the lower order terms. We have proved \cite[Lemma 4.1]{SUV14} under our settings, now Theorem \ref{local 2-tensor} follows by an argument similar to that of \cite[Section 4]{SUV14}.


\section{Proof of the global result}

We prove the part (a) of Theorem \ref{global result} based on the local result Theorem \ref{local 1-form} in this section, part (b) follows in a similar way by applying Theorem \ref{local 2-tensor}. A similar argument for the geodesic ray transform can be found in \cite{PSUZ16}. We first prove the following weaker version of Theorem \ref{global result} up to a set of empty interior. 
We define $\Sigma_t:=\tau^{-1}(t)$, $M_t:=M\setminus \{\tau\leq t\}$ and $\Omega_t:=\p M\setminus \{\tau\leq t\}$. 

\begin{Lemma}\label{global up to M_a}
Under the assumptions of Theorem \ref{global result}, there exists $v\in C^{\infty}(M_a)$ with $v|_{\Omega_a}=0$ such that $f=\bold dv$ in $M_a$. 
\end{Lemma}

\begin{proof}
Let 
\begin{align*}
\sigma:=\inf\{t\leq b: &\; \exists\,
v\in C^{\infty}(M_t), \,\mbox{such\,that} \,v|_{\Omega_t}=0,\,\mbox{and}\, f=\bold dv\,\mbox{in}\, M_t\}.
\end{align*}
We claim that $\sigma\leq a$ and we will argue by contradiction. 

First we show that $\sigma<b$. It is not difficult to see that $\Sigma_b$ is a compact subset of $\p M$ (in fact, if $\Sigma_b$ contains interior points, $\{\tau\leq b\}$ can not cover $M$). Since $\p M$ is strictly magnetic convex, by Theorem \ref{local 1-form}, for each $p\in \Sigma_b$, there is a neighborhood $O_p\subset M$ of $p$ and $v_p\in C^{\infty}(O_p)$ with $v_p|_{O_p\cap \p M}=0$ such that $f=\bold dv_p$ in $O_p$. If $O_p\cap O_q\neq \emptyset$ for some $p,q\in \Sigma_b$, we have
$$\bold d(v_p-v_q)|_{O_p\cap O_q}=0, \quad v_p-v_q|_{O_p\cap O_q\cap \p M}=0.$$
This implies that $v_p=v_q$ in $O_p\cap O_q$. Since $\tau^{-1}(b)$ is compact, there exist $t_0<b$ and $v$ smooth in $M_{t_0}$ such that $f=dv$ in $M_{t_0}$, in particular $v=v_p$ in $M_{t_0}\cap O_p$. Thus $\sigma\leq t_0<b$.

Indeed the infimum in the definition of $\sigma$ is a minimum. Let $\{t_j\}_{j=1}^{\infty}\subset (\sigma,b]$ be a strictly decreasing sequence with $t_j\to \sigma$ as $j\to \infty$. For each $j$, there is $v_j$ satisfying $f=\bold dv_j$ in $M_{t_j}$ and $v_j|_{\Omega_{t_j}}=0$. Since $\Sigma_t\cap M^{int}$ is strictly magnetic convex for any $t>a$, one can easily show that given arbitrary $k>0$, $v_{k}=v_{\ell}$ on $M_{t_k}$ for any $\ell>k$. This implies that the set $\{v_{j}\}$ defines a smooth function $v_{\sigma}$ in $M_\sigma$ with $v_\sigma |_{M_{t_j}\setminus M_{t_{j-1}}}=v_j$, $f=\bold dv_{\sigma}$ in $M_\sigma$ and $v_{\sigma}|_{\Omega_\sigma}=0$, i.e. $\sigma$ is a minimum.

Assume that $\sigma>a$ and consider the level set $\Sigma_\sigma$. There exists $v_\sigma\in C^{\infty}(M_\sigma)$ with $v_\sigma|_{\Omega_\sigma}=0$ such that $f=\bold dv_\sigma$ in $M_\sigma$. We first extend $v_\sigma$ a little bit near the boundary. Notice that $\p M$ is strictly magnetic convex and $\Sigma_\sigma\cap \p M$ is compact, by Theorem \ref{local 1-form}  and an argument similar to the one showing $\sigma<b$, one can find a neighborhood $O$ of $\Sigma_\sigma\cap \p M$ and $v_O\in C^{\infty}(O)$ such that $f=\bold dv_O$ and $v_O|_{\p M\cap O}=0$. Moreover, on the overlap $O\cap M_\sigma$, one can similarly show that $v_\sigma=v_O$ by choosing $O$ appropriately. This implies that we can actually define a smooth function $u$ on $U:=M_\sigma\cup O$. Thus now $f=\bold du$ in $U$, $u|_{\p U\cap \p M}=0$. This will allow us to avoid the set $\Sigma_\sigma\cap \p M$ for the rest of the proof.

With $U$ chosen as above, we see that $K:=\p U\cap M^{int}\cap \Sigma_\sigma$ is a compact subset of $\Sigma_\sigma\cap M^{int}$. Apply Theorem \ref{local 1-form} again, there exist $c>0$ small ($\sigma-c>a$) and an open neighborhood $V$ of $K$ in $\{\tau\leq \sigma\}\cap M^{int}$ such that the local invertibility of $I$ holds on $V$ and $(\{\sigma-c\leq \tau\leq \sigma\}\setminus O)\subset V$ (notice that $O$ is an open neighborhood of $\Sigma_\sigma\cap \p M$). In particular, the constant $c$ (which is related to the definition of the neighborhood for the local theorem) is uniform for $p\in \Sigma_t$ close to $K$ when $t$ is sufficiently close to $\sigma$, e.g. $|\sigma-t|\ll c$. Thus we pick $\sigma'>\sigma$ with $\sigma'-\sigma\ll c$, then there exists an open neighborhood $V'$ of $\Sigma_{\sigma'}\setminus O$ (compact) in $\{\tau\leq \sigma'\}\cap M^{int}$ such that the local invertibility holds in $V'$ and $(\{\sigma'-c\leq \tau \leq \sigma'\}\setminus O)\subset V'$. Obviously $\sigma'-c<\sigma$.


Now let $\phi$ be a smooth cut-off function on $M$, which satisfies $\phi\equiv 1$ near $\overline M_{\sigma'}$, supp$\phi\subset M_\sigma$, so $\phi u$ is well-defined on $M$. We denote $\tilde f=f-\bold d(\phi u)$, which is supported in $\{\tau<\sigma'\}$, by assumption $I\tilde f=0$. So we apply Theorem \ref{local 1-form} again to conclude that there is 
a smooth function $\tilde v$ defined in $V'$, such that $\tilde f=\bold d\tilde v$ in $V'$ and $\tilde v|_{V'\cap \Sigma_{\sigma'}}=0$. Moreover, on the overlap $V'\cap M_\sigma$, since $(1-\phi)u=\tilde v=0$ on $V'\cap \Sigma_{\sigma'}$, one easily obtains that $(1-\phi)u=\tilde v$ on the overlap too. Therefore, we get a smooth function $w$ on $U\cup V'$ with $f=\bold dw$ there and $w|_{\p M\cap U}=0$. In particular, this implies that $\sigma\leq \sigma'-c<\sigma$,
which is a contradiction. Thus $\sigma\leq a$ and the lemma is proved.
\end{proof}

\medskip

\begin{proof}[Proof of Theorem \ref{global result}] Note that the foliation condition implies that $M_a$ is non-trapping. On the other hand, since $\{\tau\leq a\}$ is non-trapping too, $M=M_a\cup \{\tau\leq a\}$ is non-trapping. As $\p M$ is strictly magnetic convex, by an argument similar to \cite[Proposition 5.2]{PSU12}, which is for the geodesic case, there exists $u\in C^{\infty}(SM)$ satisfying the following transport equation
\begin{equation}\label{transport equation on SM}
\bold G_{\mu}u=-f, \quad u|_{\p SM}=0.
\end{equation}
Thus by Lemma \ref{global up to M_a}
\begin{equation}\label{transport equation up to M_a}
\bold G_{\mu}(u+v)=0\,\, \mbox{in}\,\, M_a,\quad u+v|_{\p SM^{\Omega_a}}=0,
\end{equation}
where 
$$\p SM^{\Omega_a}:=\{(z,\xi)\in \p SM: z\in \Omega_a\}.$$
Since $\Sigma_t\cap M^{int}$ is strictly magnetic convex for $t\in (a,b]$, given arbitrary $(z,\xi)\in SM_a$, we can find a magnetic geodesic segment $\gamma:[0,T]\to M$ connecting $z$ with $\Omega_a$, which is completely contained in $M_a$, such that $(z,\xi)$ is either $(\gamma(0),\dot\gamma(0))$ or $(\gamma(T),\dot\gamma(T))$. Together with \eqref{transport equation up to M_a}, this implies that $u+v=0$ in $SM_a$, i.e. $u=-v$ is a smooth function on $M_a$. However, as $u\in C^{\infty}(SM)$ and the set $\{\tau\leq a\}$ has empty interior, we conclude that $u\in C^{\infty}(M)$. To show this, we take use of the spherical harmonics expansion of $u$ through the vertical Laplacian $\overset{v}{\Delta}$ on $SM$ as 
$$u=\sum_{k=0}^\infty u_k,$$
where each $u_k\in C^{\infty}(SM)$ satisfies $\overset{v}{\Delta} u_k=k(k+n-2)u_k$ ($n=\,\mbox{dim}\,M$). Note that this is an orthogonal decomposition of $u$ under the $L^2$ inner product, see e.g. \cite{PSU15} for more details. In particular, if $u\in C^{\infty}(M)$, then $u_k\equiv 0$ for all $k>0$. Since $u=-v$ on $M_a$, we get that $u_k=0$ on $SM_a$ for any $k>0$. Now given any $(z,v)\in S(M\setminus M_a)$, since $M\setminus M_a$ has empty interior, we can find a sequence $\{(z_j,v_j)\}_{j=1}^\infty\subset SM_a$ such that $(z_j,v_j)\to (z,v)$ as $j\to \infty$. Since $u_k(z_j,v_j)=0$ for any $j$ and $k>0$, $u_k(z,v)=0$ too for any $k>0$. Thus $u=u_0$ on $SM$, i.e. $u\in C^{\infty}(M)$. By \eqref{transport equation on SM}, $f=\bold G_{\mu}(-u)=\bold d(-u)$ on $M$ with $u|_{\p M}=0$, which completes the proof.
\end{proof}

\noindent {\it Remark:} It is possible to allow the existence of some type of trapped geodesics in the set $\{\tau\leq a\}$ under additional assumptions, which will still produce a smooth global solution to the transport equation \eqref{transport equation on SM}, see e.g. \cite[Proposition 5.5]{Gu14}.

\end{document}